\documentclass[preprint,12pt]{elsarticle}

\usepackage{amsmath}

\usepackage{amssymb}
\usepackage{mathrsfs}
\usepackage{enumerate, xspace}
\usepackage[bottom,first]{draftcopy}
\usepackage{changebar}
\usepackage{hyperref}
\usepackage{pdfsync}
\usepackage{amsthm}

\newtheorem{thmm}{Theorem}

\newtheorem{thm}{Theorem}[section]
\newtheorem{lem}[thm]{Lemma}

\newtheorem{prop}[thm]{Proposition}

\newtheorem{defin}{Definition}

\newcommand\cA{{\mathcal A}}

\newcommand\cE{{\mathcal E}}
\newcommand\cF{{\mathcal F}}

\newcommand\cH{{\mathcal H}}

\newcommand\cL{{\mathcal L}}

\newcommand\cM{{\mathcal M}}

\newcommand\cS{{\mathcal S}}

\newcommand\bC{{\mathbb C}}

\newcommand\bR{{\mathbb R}}
\newcommand\bT{{\mathbb T}}

\newcommand\bL{{\mathbb L}}

\journal{xxx}

\begin{document}

\begin{frontmatter}



\title{Diffusive limit and Fourier's law for the discrete Schr\"{o}dinger equation}


\author{Viviana Letizia}
\address{CEREMADE, UMR CNRS 7534\\
Universit\'e Paris Dauphine PSL\\
75775 Paris-Cedex 16, France}

\begin{abstract}
We consider the one-dimensional discrete linear Schr\"{o}dinger (DLS)
equation perturbed by a conservative stochastic dynamics, that changes the phase of each particles, conserving the total norm
(or number of particles).  The resulting total dynamics is a degenerate
hypoelliptic diffusion with a smooth stationary state. We will show
that the system has a hydrodynamical limit given by the solution of
the heat equation. When it is coupled
at the boundaries to two Langevin thermostats at two different
chemical potentials, we prove that
the stationary state, in the limit as $N\rightarrow \infty$, satisfies the
Fourier's law.
\end{abstract}

\begin{keyword}
hydrodynamic limit \sep Fourier's law \sep  Schr\"{o}dinger equation


\end{keyword}

\end{frontmatter}


\section{Introduction}

Discrete Schr\"{o}dinger equation, besides being viewed as a toy model for its
continuous counterparts, 
 has itself a physical application as a discrete systems: it serves as a model for electronic transport through crystals. In
the realm of the physics of cold atomic gases, the equation
is an approximate semiclassical description of bosons
trapped in periodic optical lattices, and
experimentally, discrete solitons have been observed in a nonlinear
optical array \cite{eisenberg1998discrete}.

In the past years much attention has been paid on the non linear case (DNLS) for which
the first analysis of the equilibrium statistical mechanics has been performed
 in \cite{rasmussen2000statistical}. It has been osserved \cite{rasmussen2000discrete} the relaxation of localized modes (discrete
 breathers) in the presence of phonon baths has been
 discussed in. Only recently,
 \cite{iubini2013off}, the non equilibrium properties have been
 explored, considering an open system that exchanges energy with external
reservoirs, for which the resulting
stationary states are investigated in the limit of low temperatures
and large particle densities, mapping the dynamics onto a coupled
rotator chain.

Here we are interested in proving the hydrodynamic limit and Fourier's law for the
DS equation in the simplified
linear case. The linear case equation is interesting as one of the most
commonly employed methods for solving one-dimensional quantum problems, for which many characteristics are still poorly
understood. The natural applications are in the context of solid-state physics,
which links the discrete model to realistic semiconductor quantum
wells and nanoelectric devices.

 In the present paper we study a chain of particles, for which the
 Hamiltonian dynamics is perturbed by a random continuous
 phase-changing noise. The resulting total dynamics of the
system is a degenerate hypoelliptic diffusion on the phase space,
which assures good ergodic properties, it conserves the total norm and destroy the other conservation laws. 
The system is considered under periodic boundary conditions. In the
first part of the article we will prove the hydrodynamic limit using
standard arguments. In the second part we will add an interaction
between the system and external
reservoirs, modeled by Ornstein-Uhlenbeck processes at the
corresponding chemical potentials. 
We prove that in the stationary state Fourier's law is valid for the
density flow. The main tool used in the proof is the bound of the
entropy production as in \cite{bernardin2005fourier}.
Then in order to obtain Fourier's law, we need to control the expected
values of the densities respect to the stationary measure, which
results in a bound of the expected total density proportional to the size of the system. 

The article is structured in the following way. In \autoref{sec:model} we define the dynamics. In \autoref{sec:hydro} we state and prove the result of hydrodynamic limit. In \autoref{sec:Fourier} we prove the Fourier's law by means of entropy bounds.
 
\section{The model}\label{sec:model}

Atoms are labeled by $x\in \bT_{N}$ where $\bT_{N}={1,...,N}$ is the
lattice torus of lenght N, corresponding to periodic boundary
conditions.
The configuration space is $\Omega^N=\bC^{\bT_{N}}$ and a generic
element is $\{\psi(x) \}_{x\in \bT_{N}}$, where $\psi(x)$ characterize
the amplitude of the wave function of each particle.
The Hamiltonian of the system writes
\begin{equation}
\cH_N=\sum_{x=1}^{N}\left| \psi(x)-\psi(x+1)\right|^2+\frac{1}{p+1}\sum_{x=1}^{N}|\psi(x)|^{p+1}
\end{equation}
where $|\psi(x)|^2$ is the number of particle or the ``mass'' at site $x$, 
at the boundary the conditions are fixed: $\psi_{N+1}=\psi_0=0$. 

The linear case is for $p=1$:
\begin{equation}
\cH_N=\sum_{x=1}^{N-1}\left(\psi(x)\psi(x+1)^*+\psi(x)^*\psi(x+1)\right)+2\sum_{x=1}^{N}|\psi(x)|^{2}
\end{equation}
the corresponding equations of motion are
\begin{equation}
\begin{split}
\frac{d\psi(x)}{dt}=i\frac{\partial \cH}{\partial \psi^*(x)}=-i
\triangle \psi(x)\\
\frac{d\psi^*(x)}{dt}=-i\frac{\partial \cH}{\partial \psi(x)}=+i
\triangle \psi^*(x).\\
\end{split}
\end{equation}
Here $\Delta \psi(x)= \psi(x+1) + \psi(x-1)-2\psi(x)$. 

We denote $\psi(x) = \psi_r(x) + i\psi_i(x) = |\psi(x)|
e^{i\theta(x)}$, 
and define the operator (on local functions $F:\bT_{N}\to \bC$)
\begin{equation}
  \label{eq:4}
  \partial_{\theta(x)} F(\psi) = i \psi(x) \partial_{\psi} F(\psi)=  \psi_i(x) \partial_{\psi_r(x)} F -
  \psi_r(x) \partial_{\psi_i(x)} F.
\end{equation}

We look for a stochastic perturbation which change randomly the phase of the wave
function, such that the total ''mass"
\begin{equation}
M_N(\psi)=\sum_{x\in \bT_N} |\psi(x)|^2
\end{equation}
is still a conserved quantity. The total ``mass'' is linear in the
number of particles $M_N(\psi)\sim N$.

The dynamics is described by the following system
 of stochastic differential equation for $x=1,...,N$
\begin{equation}\label{stoc}
\begin{cases}
\begin{split}
d\psi (x,t)=&-i\triangle \psi (x,t) dt-\frac{\gamma}{2}\psi(x,t) dt+i\psi(x,t)\sqrt{\gamma} dw_{x}\\
d\psi^* (x,t)&=+i \triangle
\psi^*(x,t)dt-\frac{\gamma}{2}\psi^*(x,t)dt-i\psi^*\sqrt{\gamma} dw_{x}\\
\end{split}
\end{cases}
\end{equation}
where $w_x(t)$ are real independent standard Brownian motions and $\gamma$ is the noise intensity parameter.

Let $\cL_N$ be the generator of the system. A core for $\cL_N$ is
given by the space $C^\infty(\Omega^N)$ of smooth functions on
$\Omega^N$ endowed with the product topology. On $C^\infty(\Omega^N)$,
the generator is defined by 
\begin{equation}
\cL_N=\cA_N +\cS_N
\end{equation}
where 
\begin{equation}\label{gen}
\cA_N=\sum_{x\in \bT_N}\{i\triangle
\psi^*\partial_{\psi(x)}-i\triangle\psi\partial_{\psi^*(x)}\}
\end{equation}
is the Liouville operator of a chain of interacting  
and
\begin{equation}
\cS_N=\frac{\gamma}{2}\sum_{x\in\bT_N} \partial_{\theta(x)}^2
\end{equation}
is the diffusive operator corresponding to the noise part of eq. (\ref{stoc})
 Since the total mass is conserved by the stochastic perturbation, we have 
\begin{equation}
\cS_N M_N(\psi)=0.
\end{equation}


In the infinite volume case, the family of product measures:
\begin{equation}
  \label{eq:9}
  d\mu_{\lambda}(d\psi) =\prod_{x\in \bT_N} \frac{e^{-\lambda |\psi(x)|^2}}{Z(\lambda)} d\psi
\end{equation}
are stationary for the dynamics, the parameter $\lambda>0$ correspond to
the conserved quantity of the dynamics, the total ``mass'', while
$Z(\lambda)$ is the normalization constant. Here $d\psi = \prod_{x=1}^N
d\psi(x) d\psi^*(x)$.

The Lie algebra, generated by the Hamiltonian vector field and the noise
fields, has full rank at every point of the state space $\bC^{N}$, so
the stationary measure has a smooth density. 
We denote by $\langle\cdot\rangle$ the expectation with respect to the
stationary measure.

Let us define the density of particle $x$ as 
\begin{equation}
\rho_x=|\psi(x)|^2,
\end{equation}
locally the conservation of mass generates an instantaneous current 
\begin{equation}\label{continuity}
\cL_N \rho_x=j_{x-1,x}-j_{x,x+1}
\end{equation}
with
\begin{equation}\label{cur}
j_{x,x+1}=-i\{\psi_x\psi^*_{x+1}-\psi^*_x\psi_{x+1}\}.
\end{equation}

\section{Hydrodynamic limit in the diffusive scaling}
\label{sec:hydro}
\subsection{Notation}

Let's introduce some notation
and definitions.

We will denote by $(\omega^N(t))_{t\geq
  0}=(\psi^N(t),{\psi^*}^N(t))_{t\geq 0}$ the process on the torus
$\bT_N$ whose evolution time is given by $N^2\cL_N $. The factor $N^2$
corresponds to the acceleration of time by $N^2$ in the stochastic
differential equations (\ref{stoc}). The associated semigroup is
denoted by $(S_t^N) _{t\geq 0}$. 

Fix a time $T>0$. Let $\cM_+$ be the space of finite positive measures
on $\bT_N$ endowed with the weak topology.
Consider a sequence of  probability measures $(Q_N)_N$ on
$D([0,T],\cM_+)$ corresponding to the Markov process $\pi_t^N$ defined
as the density of the empirical measure
\begin{equation}
\pi^N (\omega, du):=\frac{1}{N} \sum_{x \in\bT_N} \rho_x\delta_{x/N}(du)
\end{equation}
where $\delta_a(du)$ is the Dirac measure localized on the point $a\in
\bT_N$.
The time evolution of the empirical measure will be 
\begin{equation}
\pi^N_t:=\pi^N (\omega^N_t)=\frac{1}{N} \sum_{x \in\bT_N} \rho_x(t)\delta_{x/N}(du)
\end{equation}  
starting from $(\mu^N)_N$, a sequence of probability measures on $\Omega^N$ associated to a fixed initial deformation
profile $\rho_0: \bT^N\rightarrow [0,\infty)$.

We will assume that the system is close to a local equilibrium. More
precisely we have the following definition: 
\begin{defin}
A sequence $(\mu^N)_N$ of probability measures on $\bT_N$ is
associated to a deformation profile $\rho_0:\bT_N\rightarrow
[0,\infty)$, if for every continuous function $G:\bT_N
\rightarrow[0,\infty)$ and for every $\delta >0$
\begin{equation}
\lim_{N\rightarrow \infty}\mu^N\left[\left|\frac{1}{N} \sum_{x \in\bT_N}
  G(x/N)\rho_x-\int_{\bT_N} G(v) \rho_0(v) dv \right|>\delta\right]=0.
\end{equation}
\end{defin}

Our goal is to show that, if at a time $t=0$ the empirical measures
are associated to some initial profile $\rho_0$, at a macroscopic time
$t$ they are associated to a profile $\rho_t$ which is the solution of
an hydrodynamic equation.
\begin{thm}
Let $(\mu_N)_N$ be a sequence of probability measures on $\Omega^N$
associated to a bounded initial density profile $\rho_0$. 
Then for every $t>0$, the sequence of random measures
\begin{equation}
\pi_t^N(du)=\frac{1}{N}\sum_{x\in \bT}\rho_t(x)\delta_{x/N}(du)
\end{equation}
converges in probability to the absolutely continuous measure
$\pi_t(du)=\rho(t,u)du$
whose density is the solution of the heat equation:
\begin{equation}
\begin{cases}
&\partial_t \rho=\frac{1}{\gamma}\triangle \rho\\
&\rho(0,\cdot)=\rho_0(\cdot).
\end{cases}
\end{equation}
\end{thm}


For any function $w:\bT_N\rightarrow \mathbb{C}$, we denote $\nabla w$
the discrete gradient of $w$ defined by
\begin{equation}
(\nabla w)(x)=w(x+1)-w(x)
\end{equation}
and $\nabla^*$ is the adjoint on $\mathbb{L}^2(\bT_N)$ endowed with the
standard inner product 
\begin{equation}
(\nabla^* w)(x)=w(x-1)-w(x).
\end{equation}
The discrete Laplacian is $\triangle=-\nabla\nabla^*$.
For a discrete function $w$, $\triangle w$ is given by
\begin{equation}
(\triangle w)(x)=w(x+1)+w(x-1)-2w(x).
\end{equation}
If $G$ is a smooth local function on $\bT_N$ and $x\in \bT_N$, the
discrete gradient is related to the continuous one by:
\begin{equation}\label{eqn:grad}
(\nabla_N G)(x/N)=N \left[G(\frac{x}{N})-G(\frac{x-1}{N})
\right]= G'(x/N)+o(N^{-1})
\end{equation}
and the discrete Laplacian to the continuous one by:
 \begin{equation} \label{eqn:lapl}
(\triangle_N G)(x/N)=N \left[G(\frac{x+1}{N})-2G(\frac{x}{N})-G(\frac{x-1}{N})
\right]=G''(x/N)+o(N^{-1}).
\end{equation}







\subsection{Limit identification}

Under the empirical measure $Q^N$ for every smooth function $G:\bT_N\rightarrow
\mathbb{C}$, the quantity
\begin{equation}
\langle \pi_t^N, G\rangle=\frac{1}{N}\sum_{x\in \bT_N}G(\frac{x}{N})\rho_t(x),
\end{equation}
the noise conserves ponctually the density so
\begin{equation}
\langle \pi_t^N,G\rangle=\langle \pi_0^N,G\rangle-\int_0^t N^2
\mathcal{L}_N\langle \pi_s^N,G\rangle ds.
\end{equation}
We do now some manipulation on the integrand of the previous equation,
first using the definition of the empirical measure we have
\begin{equation}
N^2\cL_N\langle \Pi^N,G\rangle =\frac{1}{N} \sum_{x=1}^N G(x/N)N^2\cL_N\rho_t(x).
\end{equation}

Substituting in it the explicit continuity equation
(\ref{continuity}) we get
 \begin{equation}
N^2\cL_N\langle \Pi^N,G\rangle =\frac{1}{N} \sum_{x=1}^N G(x/N)N^2\cL_N(j_{x-1,x}(t)-j_{x,x+1}(t)).
\end{equation}
Now we perform a
summation by part
\begin{equation}
N^2\cL_N\langle \Pi^N,G\rangle=\frac{1}{N}\sum_x (\nabla_N G)(x/N) N j_{x,x+1}.
\end{equation}
A second summation by parts is also possible, substituting the current by his
\textit{fluctuation-dissipation} relation
\begin{equation}\label{fluct}
j_{x,x+1}=-\frac{1}{2\gamma}\cL_Nj_{x,x+1} +\frac{1}{\gamma} (\rho_t(x+1)-\rho_t(x))-\frac{1}{\gamma} (\cE_{x+1,x-1}-\cE_{x,x-2})
\end{equation}
where $\cE_{x+1,x-1}$ is given by 
\begin{equation}
\cE_{x+1,x-1}=\psi(x+1)\psi^*(x-1)+\psi^*(x+1)\psi(x-1).
\end{equation}
Then
\begin{equation}
\begin{split}
&\frac{1}{N}\sum_x (\nabla_N G)(x/N) N j_{x,x+1}= -\frac{1}{N}\sum_x
(\nabla G)(x/N) N \frac{1}{2\gamma}\mathcal{L} j_{x,x+1}+\\
&+\frac{1}{N}\sum_x (\nabla_N G)(x/N) N
\frac{1}{\gamma}\nabla(\rho_s(x+1)-\mathcal{E}_s(x+1,x-1))=\\
&=- \frac{1}{N}\sum_x
(\nabla_N G)(x/N) N \frac{1}{2\gamma}\mathcal{L} j_{x,x+1}+\\
&-\frac{1}{N}\sum_x (\triangle_N G)(x/N) 
\frac{1}{\gamma}(\rho_s(x+1)-\mathcal{E}_s(x+2,x)).\\
\end{split}
\end{equation}
we then obtain
\begin{equation}
\begin{split}
0&=\langle \pi_t^N,G\rangle-\langle
\pi_0^N,G\rangle-\int_0^t\frac{1}{\gamma N}\sum_x (\triangle_N G)(x/N) 
(\rho_s(x)-\mathcal{E}_s(x+1,x-1))ds\\
&-\frac{N}{2\gamma}\int_0^t\frac{1}{N}\sum_x (\nabla_N G)(x/N) \mathcal{L} j_{x,x+1}.\\
\end{split}
\end{equation}

It remains to study two terms, the first one is the contribution of
the energy between second neighbors and the other one is the
contribution of the dissipative term. We study now the former one
which is
\begin{equation}\label{510}
\int_0^t\frac{1}{N}\sum_x (\triangle_N G)(x/N) 
\frac{1}{\gamma}\mathcal{E}_s(x+1,x-1)ds,
\end{equation}
we remark that   
\begin{equation}
\mathcal{E}_s(x+1,x-1)=-\partial_{\theta(x+1)}j_{x+1,x-1}
\end{equation} 
and the \textit{fluctuation-dissipation} relation for the current
$j_{x+1,x-1}$ is
\begin{equation}
j_{x+1,x-1}=-\frac{1}{2\gamma}\mathcal{L}_N j_{x+1,x-1}+\frac{1}{\gamma}\nabla\{\mathcal{E}_{x+2,x-1}-\mathcal{E}_{x+1,x}\}.
\end{equation}
The commutator $[\partial_{\theta(x+1)}, \cL_N]$ is
\begin{equation}
\begin{split}
[\partial_{\theta(x+1)},
\cL_N]&=2\nabla\left[(\psi_r(x+2)\partial_{\psi_r(x+1)}-\psi_r(x+1)\partial_{\psi_i(x+2)})\right.\\
&+\left.(\psi_i(x+2)\partial_{\psi_r(x+1)}-\psi_i(x+1)\partial_{\psi_r(x+2)})\right],
\end{split}
\end{equation}
which applied to $j_{x+1,x-1}$ gives
\begin{equation}
\begin{split}
[\partial_{\theta(x+1)},\cL_N]j_{x+1,x-1}=&4\psi_i(x-1)\left(
  \psi_i(x+2)+\psi_i(x)\right)\\
&-4\psi_r(x-1)\left( \psi_r(x+2)+\psi_r(x)\right)
\end{split}
\end{equation}
The term
$\partial_{\theta(x+1)}\nabla\{\mathcal{E}_{x+2,x-1}-\mathcal{E}_{x+1,x}\}$
is
\begin{equation}
\begin{split}
\partial_{\theta(x+1)}\nabla\{\mathcal{E}_{x+2,x-1}-\mathcal{E}_{x+1,x}\}=&2\psi_i(x+1)\left(
  \psi_i(x)+\psi_i(x-2)\right)\\
&-2\psi_r(x+1)\left( \psi_r(x)+\psi_r(x-2)\right)
\end{split}
\end{equation}
so that 
\begin{equation}
\mathcal{E}_s(x+1,x-1)=\frac{1}{\gamma}\cL_N \partial_{\theta(x+1)}j_{x+1,x-1} -2\nabla(\cF_{x+2,x-1}-\cF_{x+1,x}) 
\end{equation}
where
\begin{equation}
\cF_{x+2,x-1}=2\{\psi_r(x+2)\psi_r(x-1)-\psi_i(x+2)\psi_i(x-1)\}.
\end{equation}
We substitute this last expression in (\ref{510}) and perform some manipulations
\begin{equation}
\begin{split}
\int_0^t\frac{1}{N}\sum_x (\triangle_N G)(x/N) 
&\frac{1}{2\gamma^2}\left(\cL_N \partial_{\theta(x+1)}j_{x+1,x-1}-2\nabla(\cF_{x+2,x-1}-\cF_{x+1,x})
\right) ds\\
&=\frac{1}{2\gamma^2N^2}\int_0^t\frac{1}{N}\sum_x (\triangle_N G)(x/N) 
N^2\mathcal{L}_N\cE_{x+1,x-1}(s)ds\\
&-\frac{1}{2N\gamma^2}\int_0^t\frac{1}{N}\sum_x (\nabla_N G)(x/N) 2\nabla(\cF_{x+2,x-1}-\cF_{x+1,x})ds\\
&=\frac{1}{2N^2\gamma}\frac{1}{N}\sum_x (\nabla_N G)(x/N) \left( \cE_{x+1,x-1}(t) -\cE_{x+1,x-1}(0)\right)+\\
&+\frac{1}{\gamma^2N^2}\int_0^t\frac{1}{N}\sum_x(\nabla_N^3G)(x/N) 
(\cF_{x+2,x-1}-\cF_{x+1,x}) ds+N_t^G.\\
\end{split}
\end{equation}
Where the quadratic variation of the martingale $N_t^G$ is
\begin{equation}
\begin{split}
[N_t^G]^2=&\frac{N^2}{N^6}\sum_x\int_0^t\left( (\triangle_N
  G)(x/N)\right)^2\left(\partial_{\theta(x+1)}\cE_{ x+1,x-1} \right)^2=\\
=&\frac{N^2}{N^6}\sum_x\int_0^t\left( (\triangle_N
  G)(x/N)\right)^2j_{x+1,x-1}^2. 
\end{split}
\end{equation}
So then, the contribution of the total term studied here, (\ref{510}), can be neglected
considering the following bounds 
\begin{equation}\label{bounds}
\begin{split}
&\frac{1}{N^3}\sum_{x\in \bT_N}\cE_{x,x+p}(t)\leq \frac{1}{2N^3}\sum_{x\in \bT_N}
\left(|\psi(x,t)|^2+|\psi(x+p,t)|^2\right)=\frac{M_N(\psi)}{N^3}\to 0,\\
&\frac{1}{N^2}\sum_{x\in \bT_N}\cF_{x,x+p}(t)\leq \frac{1}{2N^2}\sum_{x\in \bT_N}
\left(|\psi(x,t)|^2+|\psi(x+p,t)|^2\right)=\frac{M_N(\psi)}{N^2}\to 0,\\
&\frac{1}{N^4}\sum_{x\in \bT_N}j_{x,x+p}^2(t)\leq\frac{1}{N^4}\sum_{x\in \bT_N}
|\psi(x,t)|^4\leq \frac{1}{N^4}\left( \sum_{x\in \bT_N} |\psi(x,t)|^2\right)^2\leq
\frac{M^2_N(\psi)}{N^4}\to 0.
\end{split}
\end{equation}
We expect then the following characterization of the hydrodynamic limit:
\begin{equation}
\langle \pi_t^N,G\rangle=\langle
\pi_0^N,G\rangle+\frac{1}{\gamma N}\sum_{x\in \bT_N} \int_0^t (\triangle_N G)(x/N)\rho_s(x)ds+o(N).
\end{equation}

\subsection{A rigorous proof}

Let $G\in C^2(\bT_N)$, then under $Q_N$ the quantity
\begin{equation}
\langle  \pi_t^N, G\rangle=\frac{1}{N}\sum_{x\in\bT} G(\frac{x}{N})\rho_t(x)
\end{equation}
has an associated process
\begin{equation}\label{marti}
\langle  \pi_t^N, G\rangle=\langle  \pi_0^N,
G\rangle+\int_0^t (\partial_s+N^2 L_N) \langle  \pi_s^N, G\rangle ds
\end{equation}
with respect to the filtration
$\cF_t=\sigma(\rho_s , s\leq t)$.


In order to prove the convergence of the sequence, we need first to
show its relatively compactness, then that all
converging subsequences converge to the same limit. 

\subsection{Relative Compactness}

To show that $(Q_N)_N$ is relatively compact, it suffices to
prove that the sequence of laws of the real processes $(\langle
\pi_t^N, G\rangle)_{t\geq 0}$ is relatively compact for any fixed G in
$C^2(\bT_N)$. We can repeat the same
arguments as in \cite{kipnis1999scaling} (Theorem $2.1$, pag. 55).
Let us denote $Q_N^G$ the probability $Q_NG^{-1}$ on $C([0,T],\bR)$,
and define for any function $x\in C([0,T],\bR)$ and any $\delta>0$,
the modulus of continuity of $x$ by $w(x,\gamma)=\sup\{ |x(s)-x(t)|;
s,t \in [0,T], |s-t|\leq \delta\}$.
The criterion for relative compactness of probabilities is: 
\begin{lem}
The sequence $Q_N^G$ is relatively compact if 
\begin{itemize}
\item $\forall t \in [0,T]$, $\forall \epsilon >0$ $\exists
  A=A(t,\epsilon)>0$, $\sup_N Q_N^G[|\langle  \pi_t^N, G\rangle|\geq A]\leq \epsilon $   
\item $\lim\sup_{\delta\rightarrow 0} \limsup_{N\rightarrow \infty}Q_N^G[w(\langle  \pi^N, G\rangle,\delta)>\epsilon]=0$
\end{itemize}
\end{lem}
\begin{proof}
The first condition of the lemma is satisfied thanks to the
conservation of the total ``mass'' and the following bound
\begin{equation}
|\langle \pi_t^N, G\rangle|\leq \|G\|_{\infty} \frac{1}{N}\sum_{x\in
  \bT_N}\rho_t(x)=\|G\|_{\infty}  \frac{M_N(\psi)}{N}\leq C(G)
\end{equation}
where $C(G)$ is a constant depending only on G.
Then
\begin{equation}
\begin{split}
 Q_N^G\left[ |\langle  \pi_t^N, G\rangle|\geq A\right] &=Q_N^G\left[ |\langle  \pi_0^N,
 G\rangle+\frac{1}{\gamma N}\sum_{x\in \bT_N} \int_0^t (\triangle_N G)(x/N)\rho_s(x)ds|\geq A \right]\leq \\
&\leq \frac{1}{A}E_{\mu_N}\left[ |\langle  \pi_0^N, G\rangle+\frac{1}{\gamma N}\sum_{x\in \bT_N} \int_0^t (\triangle_N G)(x/N)\rho_s(x)ds| \right]\\
&\leq \frac{C(G,t)}{\gamma A}.
\end{split}
\end{equation}
The first condition is satisfied choosing $A \geq \frac{C(G,t)}{\gamma \epsilon }$.
Also the second condition is verified:
\begin{equation}
\begin{split}
Q_N^G[\sup_{|t-s|\leq \delta}|\langle  \pi_t^N, G\rangle-\langle
\pi_s^N, G\rangle|]&\leq \frac{1}{\epsilon\gamma N}
E_{\mu_N}[\sup_{|t-s|\leq \delta}|\int_s^t\sum_{x\in\bT_N}(\triangle_N G)(x/N)\rho_u(x)du|]\\
&\leq \frac{C(G)\delta M_N(\psi)}{N\gamma \epsilon}
\end{split}
\end{equation}
which goes to zero for $N\rightarrow \infty$ and $\delta \rightarrow 0$.
\end{proof}
\subsection{Uniqueness of limit points}

After proving the relatively compactness of $(Q_N)_N$, we want to characterize the
limit points of $Q_N$.

\begin{lem}Let $Q^*$ be a limit point of the sequence $(Q_N)_N$, then $Q^*$ is
concentrated on trajectories $\pi_t \in C([0,T],\cM)$ satisfying
\begin{equation}
\langle  \pi_t, G\rangle=\langle  \pi_0,
G\rangle+\frac{1}{\gamma}\int_0^t \langle  \pi_s, G''\rangle ds
\end{equation}
\end{lem}
\begin{proof}
Let $Q^*$ be a limit point and let $Q_{N_k}$ be a sub-sequence
converging to $Q^*$. 
We can replace de discrete Laplacian by the continuous one, since
$(\triangle_N G)(x/N)=G''(x/N)+o(N^{-1})$, uniformly in $N$,
in eq. (\ref{marti}). We fix $t\in [0,T]$. The application from $
C([0,T],\cM)$ to $\bR$, which associates $|\langle  \pi_t, G\rangle-\langle  \pi_0,
G\rangle-\frac{1}{\gamma}\int_0^t \langle  \pi_s, G''\rangle
ds|$ to a path $\{ \pi_t; 0\leq t\leq T\}$, is continuous. So  
\begin{equation}
\begin{split}
\lim\inf_{k\rightarrow \infty} & Q_{N_k}\left( \langle | \pi_t, G\rangle-\langle  \pi_0,
G\rangle-\frac{1}{\gamma}\int_0^t \langle  \pi_s, G''\rangle
ds|> \epsilon \right)\\
&\geq Q^*\left( \langle | \pi_t, G\rangle-\langle  \pi_0,
G\rangle-\frac{1}{\gamma}\int_0^t \langle  \pi_s, G''\rangle
ds|> \epsilon \right)\\
\end{split}
\end{equation}
since the set is open.
 Then simply observing that
\begin{equation}
Q_N[\sup_{0\leq t\leq T} |M_t^G|]=0,
\end{equation}
we can conclude that all limit points $Q^*$ are concentrated on
trajectories $\pi_t$ satisfying 
\begin{equation}
\langle  \pi_t^N, G\rangle=\langle  \pi_0^N,
G\rangle+\frac{1}{\gamma}\int_0^t \langle  \pi_s^N, G''\rangle ds
\end{equation}

It remains to prove that the limit trajectories are absolutely continuous
respect to the Lebesgue measure. 
\begin{lem}
All limit points $Q^*$ of $(Q_N)_N$ are concentrated on absolutely
continuous measures, with respect to the Lebesgue measure,
$\pi(du)=\rho(u)du$ such that $\pi \in \bL^2(\bT_N, du):$
\begin{equation}
Q^*\left\{\pi:\pi(du)=\rho(u)du\right\}=1
\end{equation}
\end{lem}
\begin{proof}
Since 
\begin{equation}
 Q^*\left( \langle | \pi_t, G\rangle-\langle  \pi_0,
G\rangle-\frac{1}{\gamma}\int_0^t \langle  \pi_s, G''\rangle
ds|=0 \right)=1
\end{equation}
then choosing $\pi_0=\rho_0(u)du$ it implies that $\pi_t=\rho_t(u)du$.
\end{proof}
 

\end{proof}

\subsection{Uniqueness of weak solutions of the heat equation and convergence in probability at fixed time}

We need to show that there exists only one weak solution of the
hydrodynamic equation. There are different methods to prove that there
exists only one weak solution of the heat equation. We refer to
\cite{kipnis1999scaling} for the proof.

The limiting probability measure is concentrated on weakly continuous
trajectories, thus $\pi_t^N$ converges in distribution to the
deterministic measure $\pi_t(u) du$. Since convergence in distribution
to a deterministic variable implies convergence in probability, the
theorem is proved.

\section{Physical implications}
\label{sec:Fourier}

The model is composed of $x\in \{1,...,N\}$ atoms attached at their
extremities to particle reservoirs of Langevin type at two different densities $\mu_l$
and $\mu_r$. The interaction between the reservoirs is modeled by two
Ornstein Uhlenbeck process at the corresponding chemical potentials.

The stationary state is given by the law of independent Gaussian
variables if the two reservoirs have the same chemical potentials.

We prove that the Fourier's law is valid in the stationary state for
the density flow, that the total mass is proportional to its size and
that the average density per volume, in the infinite volume limit is
given by the average of the chemical potentials at the boundaries.

We attach the first particle $1$ and the last $N$ to two Langevin
 thermostats, the dynamics is then described by the following system
 of stochastic differential equation
\begin{equation}
\begin{cases}
\begin{split}
d\psi (x,t)=&-i\triangle \psi (x,t) dt-\frac{\gamma}{2}\psi(x,t) dt+i\psi(x,t)\sqrt{\gamma} dw_{x}\\
d\psi^* (x,t)&=+i \triangle
\psi^*(x,t)dt-\frac{\gamma}{2}\psi^*(x,t)dt-i\psi^*\sqrt{\gamma} dw_{x}\\
&x=2,...,N-1\\
d\psi(1,t) =&-i\triangle \psi (1,t) dt-\frac{1}{2} (\delta+\gamma)\psi
(1,t)dt+i\psi\sqrt{\gamma} dw_{1}+\sqrt{\delta \mu_l}dw_{\mu_l,1}\\
d\psi^*(1,t) &=+i \triangle
\psi^*(1,t)dt-\frac{1}{2} (\delta+\gamma)\psi^*(1,t)dt-i\psi^*(1,t)\sqrt{\gamma} dw_x +\sqrt{\delta\mu_l}dw^*_{\mu_l,1}\\
d\psi(N,t) =&-i\triangle \psi (N,t) dt-\frac{1}{2} (\delta+\gamma)\psi(N,t) dt+i\psi\sqrt{\gamma} dw_{N}+\sqrt{\delta\mu_r}dw_{N,\mu_r}\\
d\psi^* (N,t)&=+i \triangle
\psi^*(N-1,t)dt-\frac{1}{2} (\delta+\gamma)\psi^*(N-1,t)dt-i\psi^*(N)\sqrt{\gamma} dw_{N}+\sqrt{\delta\mu_r}dw^*_{N,\mu_r}\\
\end{split}
\end{cases}
\end{equation}
Where $w_x(t)$ are real independent standard Brownian motions, and
$w_{0,1}(t)$ and $w_{N-1,N}(t)$ are complex independent standard
Brownian motions.

The generator of the dynamics is $\cL=\cL_N+\cL_L+\cL_R$ where $\cL_N$ is (\ref{gen}),
and 
\begin{equation}
\begin{split}
\cL_L=&+\frac{\delta}{2}\{\mu_l
(\partial^2_{\psi_{r(1)}}+\partial^2_{\psi_{i(1)}})-(\psi_{r(1)}\partial_{\psi_{r(1)}}+\psi_{i(1)}\partial_{\psi_{i(1)}})
\},\\
\cL_R=&\frac{\delta}{2}\{\mu_r
(\partial^2_{\psi_{r(N)}}+\partial^2_{\psi_{i(N)}})-(\psi_{r(N)}\partial_{\psi_{r(N)}}+\psi_{i(N)}\partial_{\psi_{i(N)}})
\}\\
\end{split}
\end{equation}
The currents are 
\begin{equation}\label{cur}
\begin{split}
j_{x,x+1}&=-i\{\psi_x\psi^*_{x+1}-\psi^*_x\psi_{x+1}\} \; \mbox{for}\; x=2,...,N-1,\\
j_{0,1}&=(2\mu_l-\rho_1),\\
j_{N,N+1}&=-(2\mu_r-\rho_{N})
\end{split}
\end{equation}
Because of the presence of reservoirs we have stationarity, for any $x=1,..,N-1$, we have
\begin{equation}\label{statio}
\langle j_{x,x+1}\rangle=\langle j_{0,1}\rangle=\langle j_{N-1,N}\rangle.
\end{equation}

\subsection{Entropy production}

Denote by $g_{\mu_r}(\psi_1,\psi_1^*,...\psi_{N},\psi_{N}^*)$ the
density of the product of Gaussians with mean $0$ and variance $\mu_r$
\begin{equation}
g_{\mu_r}(\psi_1,\psi_1^*,...\psi_{N},\psi_{N}^*)=e^{-\sum_{x=1}^{N}\frac{|\psi(x)|^2}{2\mu_r}}
\end{equation}
and by $f_N$ the density of the stationary measure with respect to
$g_{\mu_r}$
\begin{equation}
\langle \cdot \rangle=\int f_N g_{\mu_r} d\psi
\end{equation}
where $d\psi=\prod_{x=1}^{N}d\psi(x) d\psi^*(x)$, by hypoellipticity
this density is smooth.
By stationarity
\begin{equation}
\begin{split}
0&=-2\langle \cL_N \log f_N \rangle\\
&= \gamma \sum_{x=1}^{N}\int \frac{(\partial_{\theta(x)}f_N)^2}{f_N}
g_{\mu_r} d\psi+\frac{\delta}{2}\mu_r\int \frac{(\partial_{\psi(N)}f_N)^2}{f_N}g_{\mu_r} d\psi-2\langle \cL_l \log f_N \rangle\\
\end{split}
\end{equation}
for the left thermostat entropy production, let
$h=g_{\mu_l}/g_{\mu_r}$ and we rewrite the last
term as
\begin{equation} 
\begin{split}
-2\langle \cL_l \log f_N \rangle=& -2 \int \frac{f_N}{h}\cL_l \log(
\frac{f_N}{h})g_{\mu_l} d\psi-2 \int \frac{f_N}{h}\cL_l \log(h)g_{\mu_r}
d\psi \\
&=\frac{\delta}{2}\mu_l\int\frac{(\partial_{\psi(1)}(f_N/h))^2}{f_N/h}g_{\mu_l}d\psi+\delta(\mu_l-\mu_r)(2\mu_l-\langle
\rho_1\rangle)
\end{split}
\end{equation}
Recognizing the last term as the current
$\langle j_{0,1}\rangle=(2\mu_l-\langle\rho_1\rangle)$ we can have the following bound
\begin{equation}
\begin{split}
& \gamma \sum_{x=1}^{N}\int \frac{(\partial_{\theta(x)}f_N)^2}{f_N}
g_{\mu_r} d\psi+\\
&+\frac{\delta}{2}\mu_r\int \frac{(\partial_{\psi(N-1)}f_N)^2}{f_N}g_{\mu_r}
d\psi+\frac{\delta}{2}\mu_l\int\frac{(\partial_{\psi(1)}(f_N/h))^2}{f_N/h}g_{\mu_l}d\psi=\\
&=\delta (\mu_l-\mu_r)\langle j_{x,x+1}\rangle\geq 0\\
\end{split}
\end{equation}

The right sign for the density current is then $\langle
j_{x,x+1}\rangle \leq0$ (resp. $\langle
j_{x,x+1}\rangle \geq 0$) if $\mu_l\leq \mu_r$ (resp. $\mu_l\geq\mu_r$).

\subsection{Scaling of the average current}

In order to recover the Fourier's law we need to bound the
instantaneous current.
From the stationarity, (\ref{cur}) and (\ref{statio}), we have
\begin{equation}\label{bordi}
\langle\rho_1\rangle+ \langle\rho_{N}\rangle=2( \mu_l+\mu_r)
\end{equation}
By (\ref{fluct}) we have:
\begin{equation}\label{fluct-diss}
\begin{split}
j_{x,x+1}&=\frac{1}{\gamma}\{
(\rho_{x+1}-\rho_x)-\frac{1}{2}(\cE_{x,x+2}-\cE_{x-1,x+1})\}-\frac{1}{2\gamma}L_Nj_{x,x+1}\mbox{\quad
for }x=2,...,N-2,\\
j_{1,2}&=\frac{4}{4\gamma+\delta}\{
(\rho_2-\rho_1)-\frac{2}{4\gamma+\delta}\cE_{1,3}-\frac{2}{4\gamma+\delta}L_Nj_{1,2}\\
j_{N-1,N}&=\frac{4}{4\gamma+\delta}(\rho_N-\rho_{N-1})+\frac{2}{4\gamma+\delta}\cE_{N,N-2}-\frac{2}{4\gamma+\delta}L_Nj_{N-1,N},
\end{split}
\end{equation}
where
\begin{equation}
\cE_{x,x+2}=\psi(x)\psi^*(x+2)+\psi^*(x)\psi(x+2).
\end{equation}

Using the stationarity of the current we obtain 
\begin{equation}\label{stat}
\begin{split}
\langle j_{x,x+1}  \rangle&=\frac{1}{N-3}\sum_{x=2}^{N-2}\langle j_{x,x+1}
\rangle\\
&=\frac{1}{(N-3)\gamma}(\langle\rho_{N-1}\rangle -\langle\rho_2\rangle)-
\frac{1}{2(N-3)\gamma}\langle ( \cE_{N-2,N}-\cE_{1,3})\rangle\\
\end{split}
\end{equation}
and by the relation $\langle j_{1,2}\rangle=\langle j_{N-1,N}\rangle$,
we get
\begin{equation}
\langle\rho_{N-1}\rangle =-\langle\rho_2\rangle+2(\mu_l+\mu_r)+\frac{1}{2}\langle (\cE_{N-2,N}-\cE_{1,3})\rangle.
\end{equation}
we substitute the expression for $\langle\rho_{N-1}\rangle$ in
(\ref{stat}) and obtain
\begin{equation*}
\langle j_{x,x+1}  \rangle=\frac{1}{(N-3)\gamma}(2(\mu_l+\mu_r)
-2\langle\rho_2\rangle +\langle\cE_{1,3}\rangle).
\end{equation*}
By (\ref{fluct-diss}) we get $\langle\cE_{1,3}\rangle$ as function of
the densities and currents
\begin{equation}
\langle\cE_{1,3}\rangle=-\frac{4\gamma+\delta}{2}\langle j_{1,2}
\rangle+2(\langle\rho_2\rangle -\langle \rho_1\rangle ),
\end{equation}
and then
\begin{align*}
\langle j_{x,x+1}  \rangle=\frac{1}{(N-3)\gamma}\left(2(\mu_l+\mu_r)
-2\langle\rho_2\rangle -\frac{4\gamma+\delta}{2}\langle j_{x,x+1}
\rangle+2(\langle\rho_2\rangle -\langle \rho_1\rangle \right)\\
=\frac{1}{(N-3)\gamma}\left(2(\mu_l+\mu_r)
-\frac{4\gamma+\delta}{2}\langle j_{x,x+1}
\rangle -\langle \rho_1\rangle \right)\\
= \frac{2}{\gamma (N-3)+4\gamma +\delta}\left(
  (\mu_l+\mu_r)-\langle \rho_1 \rangle\right).
\end{align*}
Given that $\rho\geq 0$, we can bound the current by the external
chemical potential as
\begin{equation}
\langle j_{x,x+1}\rangle \leq \frac{2   (\mu_l+\mu_r)}{\gamma (N-3)+4\gamma +\delta}.
\end{equation}
So there exists a constant $C$, which depends on
 $\mu_r$ and $\mu_l$,
  such that
$$
\langle j_{x,x+1}\rangle \leq \frac{C}{N}.
$$
for $\mu_l >\mu_r$, and 
$$\langle j_{x,x+1}\rangle\geq - \frac{C}{N}$$
for $\mu_l >\mu_r $, such that 
\begin{equation}\label{limcurr}
|\langle j_{x,x+1}\rangle| \leq \frac{C}{N}.
\end{equation}.
Thanks to this bound to the current we are able now to claim the result in the following theorems.
\begin{thmm}
For any $\gamma>0$
\begin{equation}
\lim_{N\rightarrow\infty}N\langle j_{x,x+1}\rangle=\frac{2}{\gamma}(\mu_r-\mu_l)
\end{equation}
\end{thmm}

\begin{thmm}
\begin{equation}
\lim_{N\rightarrow\infty}\frac{\langle M_N(\psi)\rangle}{N}=(\mu_r+\mu_l)
\end{equation}
\end{thmm}

\subsection{Fourier's law}

\begin{prop}
For $x=1$ and $x=N-2$ we have
\begin{equation}\label{prop1}
\lim_{N\rightarrow \infty}\langle \psi_x \psi_{x+2}^*\rangle=0
\end{equation}
\end{prop}
\begin{proof}
The proof for $x=1$ and $x=N-2$ are similar. Let's do it for $x=1$. We make use of Cauchy-Schwarz inequality.
\begin{equation}
\begin{split}
\langle \psi_1 \psi_{3}^* \rangle &=\int \psi_1 \psi_3^*
\frac{f_N}{h}g_{\mu_l}d\psi d\psi^*=\\
&=\mu_l\int \psi_3^*\partial_{\psi_1^*}(\frac{f_N}{h})g_{\mu_l}d\psi
d\psi^*\\
&\leq \mu_l\sqrt{\langle \rho_3\rangle}\sqrt{\int\left( \frac{\partial_{\psi_1^*}(f_N/h)}{f_N/h}\right)^2g_{\mu_l}d\psi
d\psi^*}\\
&\leq\frac{C}{\sqrt{N}}.
\end{split}
\end{equation}
\end{proof}
\begin{prop}
For $x=1$ and $x=N-1$
\begin{equation}\label{prop2}
\lim_{N\rightarrow\infty}(\langle \rho_x\rangle-\langle \rho_{x+1} \rangle)=0.
\end{equation}
\end{prop}
\begin{proof}
By (\ref{fluct-diss})
\begin{equation}
\gamma \langle j_{12}\rangle +\frac{1}{2} \langle(\cE_{2,4}-\cE_{1,3}) \rangle=\langle \rho_2\rangle -\langle \rho_1 \rangle
\end{equation}
 then by (\ref{limcurr}) and (\ref{prop1}) we have
\begin{equation}
\lim_{N\rightarrow \infty}(\langle \rho_1\rangle-\langle \rho_2 \rangle)=0
\end{equation}
and similarly for $x=N-1$.
\end{proof}
Then we have
\begin{equation}
\begin{split}
&\lim_{N\rightarrow\infty}\langle \rho_1\rangle=2\mu_l\\
&\lim_{N\rightarrow\infty}\langle \rho_N\rangle=2\mu_r
\end{split}
\end{equation}
and the Fourier's law is
\begin{equation}
\begin{split}
\lim_{N\rightarrow \infty} N \langle j_{x,x+1}\rangle = \frac{2}{\gamma}(\mu_r-\mu_l).
\end{split}
\end{equation}

\subsection{Average ``mass" density}

We define a function
\begin{equation}
\phi(x)=\frac{1}{\gamma}(\langle \rho_x\rangle-\frac{1}{2}\langle \cE_{x-1,x+1} \rangle),
\end{equation}
by the continuity equation for $x=2,...,N-1$
\begin{equation}
\cL\rho_x=j_{x-1,x}-j_{x,x+1}
\end{equation}
and the fluctuation-dissipation equation
\begin{equation}
\langle j_{x,x+1}\rangle=-\nabla^* \phi(x)
\end{equation}
we can write
\begin{equation}
\triangle \phi(x)=0 \; \mbox{  for } x=2,...,N-1.
\end{equation}
By the discrete maximum principle 
$ |\phi(x)|\leq max\{ \phi(2),\phi(N-1)
\}$
and using the definition of currents (\ref{fluct-diss})
\begin{equation}
\begin{split}
\phi(2)&=(\frac{4 \gamma+ \delta}{4\gamma})\langle j_{12}\rangle
+\langle \rho_1\rangle\\
\phi(N-1)&=-(\frac{4 \gamma+ \delta}{4\gamma})\langle j_{N-1,N}\rangle
+\langle \rho_N\rangle\\
\end{split}
\end{equation}
both with (\ref{bordi}), so, given that $\langle \rho_x\rangle \leq
2(\mu_l+\mu_r)$ for $x=1,N-1$  
it follows 
$$|\phi(x)|\leq \frac{4 \gamma+ \delta}{4\gamma}|\langle j_{x,x+1}\rangle|
+2(\mu_l+\mu_r)$$
In view of (\ref{limcurr}), it follows that
\begin{equation}
|\phi(x)|\leq \frac{C}{N}+2(\mu_l+\mu_r)  \; \mbox{  for } x=1,...,N-1.
\end{equation}
Furthermore, given the results of the previous section, propositions
\ref{prop1} and \ref{prop2}, and the explicit expression of $\phi$
\begin{equation}
\lim_{N\rightarrow \infty} \frac{1}{N}\sum_{x=1}^{N-1}\phi(x)=\frac{1}{\gamma}(\mu_l+\mu_r):
\end{equation}
then for $$M_N(\psi)=\sum_{x=1}^{N-1}\rho_x$$
we obtain the result
\begin{equation}
\lim_{N\rightarrow\infty}\frac{1}{N}\sum_{x=1}^{N-1}\langle M_N(\psi)\rangle=(\mu_l+\mu_r).
\end{equation}

We can verify that, at equilibrium, the two thermostats must
have the same chemical potentials.

 \begin{prop}
If $\langle j_{x,x+1} \rangle=0$ then $\mu_l=\mu_r$. 
\end{prop}
\begin{proof}
By the stationarity and  by (\ref{cur}), we can write the average
densities at the extremities in $x=1$ and $x=N$
\begin{equation}
\begin{split}
\langle \rho_1 \rangle& = 2\mu_l\\
\langle \rho_N \rangle &=2\mu_r
\end{split}
\end{equation}
and by (\ref{fluct-diss}) in the bulk, $x=1,...,N-1$
\begin{equation}
\langle \rho_{x+1}\rangle-\langle \rho_{x}\rangle
=\frac{1}{2}\left( \langle \cE_{x,x+2}\rangle -\langle \cE_{x-1,x+1} \rangle\right) 
\end{equation}
Then substituting recursively the extremity density value in $1$, we find
\begin{equation}
\langle \rho_x \rangle =2\mu_r +\frac{1}{2}\langle \cE_{x-1,x+1}\rangle
\end{equation}  
and similarly when substituting the $N$ density value
\begin{equation}
\langle \rho_x \rangle =2\mu_l +\frac{1}{2}\langle \cE_{x-1,x+1}\rangle
\end{equation} 
so that $\mu_r=\mu_l$.
\end{proof}

\subsection{Non Linear Case}

When the Hamiltonian is nonlinear, $p>1$, the current doesn't decompose in \textit{fluctuation-dissipation} terms:
\begin{equation}
j_{x,x+1}=-\cL_N j_{x,x+1}+\nabla \rho_x+\cE_{x-1,x+1}+\cE_{x,x+1}+\cE_{x,x+1}|\psi_x|^{p-1}.
\end{equation} 
Being a non gradient system a correction term in the second order approximation of a local Gibbs measure in the relative entropy method should be added. Unfortunately the non linearity made it for us impossible to find the proper correction term which would gauge the system in the local averages.





\section*{Acknowledgement}

I would like to thank Stefano Olla for posing me this problem, for the 
helpful suggestions and the interests in this work.
\nocite{*}
\bibliographystyle{plain}	
\bibliography{fourierbib}







\end{document}